\documentclass[letter,preprint]{elsarticle}

\usepackage{amsmath,amsfonts,amssymb,amsthm}
\usepackage{mathrsfs}  
\usepackage{algorithm}
\usepackage{array}
\usepackage[caption=false,font=normalsize,labelfont=sf,textfont=sf]{subfig}
\usepackage{textcomp}
\usepackage{stfloats}
\usepackage{url}
\usepackage{verbatim}
\usepackage{graphicx}

\usepackage{microtype}
\usepackage{booktabs}
\usepackage{hyperref}

\usepackage{mathtools}
\usepackage[capitalize,noabbrev]{cleveref}
\usepackage[noend]{algpseudocode}
\usepackage{epstopdf}
\usepackage{multirow}
\usepackage{bm}
\usepackage{color}
\usepackage{wrapfig}
\usepackage[textsize=tiny]{todonotes}

\newtheorem{theorem}{Theorem}
\newtheorem{corollary}{Corollary}
\newtheorem{lemma}{Lemma}
\newtheorem{remark}{Remark}
\newtheorem{proposition}{Proposition}

\newcommand{\setsep}{\,\big|\,}
\newcommand{\ptset}{\mathcal}
\begin{document}

\begin{frontmatter}
\title{On the existence and estimates of nested spherical designs}
\author[1]{Ruigang Zheng}
\ead{ruigzheng2-c@my.cityu.edu.hk}

\author[1]{Xiaosheng Zhuang\corref{cor1}}
\ead{xzhuang7@cityu.edu.hk}

\cortext[cor1]{Corresponding author}
\affiliation[1]{organization = {Department of Mathematics, City University of Hong Kong},
			city={Hong Kong SAR},
			country={China}
}


\begin{abstract}
In this paper, we prove the existence of a spherical $t$-design formed by adding extra points to an arbitrarily given point set on the sphere and, subsequently, deduce the existence of nested spherical designs. Estimates on the number of required points are also given. For the case that the given point set is a spherical $t_1$-design such that $t_1 <  t$ and the number of points is of optimal order $t_1^d$, we show that the upper bound of the total number of extra points and given points for forming nested spherical $t$-design is of order $t^{2d+1}$. A brief discussion concerning the optimal order in nested spherical designs is also given.
\end{abstract}


\end{frontmatter}

\section{Introduction and motivation}\label{sec:intro}

In function approximation on the $d$-dimensional unit sphere $\mathbb{S}^d$, approaches such as hyperinterpolation \cite{sloan1995polynomial}, multiscale analysis \cite{gia2010multiscale}, localized systems \cite{le2008localized,narcowich2006localized}, and {spherical} framelets \cite{wang2020tight,xiao2023spherical} all rely on exact quadrature rules for spherical harmonics. Generally, the approximation error decreases as the degree of spherical harmonics applied increases. On the other hand, quadrature rules of spherical harmonics with higher degrees typically require more points. Consequently, the number of points in quadrature rules becomes the main factor of computational and storage efficiency for fine approximation. Apart from the number of points of quadrature rules, another factor that incurs storage burden is the non-nested structures of quadrature rules. For framelet systems such as those in  \cite{wang2020tight,xiao2023spherical}, it is required to form approximation successively using spherical harmonics from lower to higher degrees. For each degree,  the framelet systems utilize quadrature rules that are exact for that degree (including lower degrees) but not higher ones since they do not need excessive points for exact integration. However, unlike data on Euclidean domains such as signals, images, and videos that are defined on the regular grids with easy down- and up-sampling operations (dyadic operations),  in most commonly used quadrature rules on the sphere such as the Gauss–Legendre tensor product rule \cite{GL}, the point sets for spherical harmonics of lower degree are not contained in the point sets of higher degree, that is,  the point sets of different degrees are not nested. Thus, in each phase of computation, with respect to a certain degree, one has to apply a different quadrature rule, which results in separate (extra) storage in spherical signal processing.

Among all quadrature rules on the sphere, one of the most well-known ones is the so-called \emph{spherical designs} \cite{STD},  thanks to its deep theoretical connections and wide practical impacts in many fields such as approximation theory,   statistics for experimental design, combinatorics, geometry, coding theory, and so on. The concept of spherical design was first introduced by Delsarte, Goethals, and Seidel \cite{delsarte1991spherical} in 1977. {In detail, a set $\ptset{X}_N:=\{x_1,\dots,x_N\}$} \footnote[1]{It is understood that the points in the set can be identical, i.e. $\mathcal{X}_N$ is a multi-set. This is also assumed for other point sets and spherical designs in $\mathbb{S}^d$ in what follows.}  {of  $N$-points on $\mathbb{S}^d:=\{x\in\mathbb{R}^{d+1}\setsep \, |x|=1\}$ is a \emph{spherical $t$-design} if} 
\[
	\int_{\mathbb{S}^d} P(x)d\mu_d(x) = \frac{1}{N}\sum_{i=1}^N P(x_i)
\]
for all polynomials $P$ with $d+1$ variables and $t$ degree, where $\mu_d$ is the normalized Lebesgue measure of $\mathbb{S}^d$ such that $\mu_d(\mathbb{S}^d)=1$ and $|\cdot|$ is the Euclidean norm of $\mathbb{R}^{d+1}$. One of the fundamental research topics on spherical $t$-design is to find the theoretical upper bounds for the minimal number of points $N=N(d,t)$ for a spherical $t$-design. A lower bound of $N=N(d,t)$ is provided in \cite{delsarte1991spherical} stating that
\[
	N(d,t)\geq \left\{\begin{aligned}
	\binom{d+k}{d}+\binom{d+k-1}{d} \quad &\text{if } t=2k,\\
	2\binom{d+k}{d}\quad \quad \quad \quad \quad &\text{if } t=2k+1.
\end{aligned}\right.
\]
This implies that $N(d,t)\geq c_dt^d$, where $c_d$ is a constant depending only on $d$.  The upper bound of $N(d,t)$ is subsequently proved in \cite{wagner1991averaging}, \cite{bajnok1992construction}, and \cite{korevaar1993spherical} that  $N(d,t)\leq C_dt^{Cd^4}$, $N(d,t)\leq C_dt^{Cd^3}$, and $N(d,t)\leq C_dt^{(d^2+d)/2}$, respectively,  for some constant $C_d$ depending on $d$ and some fixed constant $C$. It is also conjectured in \cite{korevaar1993spherical} that $N(d,t)\leq C_dt^d$. The upper bound is improved in \cite{bondarenko2010spherical} by showing that $N(d,t)\leq C_dt^{2d(d+1)/(d+2)}$ based on the Brouwer fixed point theorem. The conjecture is eventually proved in \cite{ostd} by applying the Brouwer degree theory. This concludes that $t^d$ is the optimal order of the spherical $t$-design. In \cite{bondarenko2015well}, the existence of well-separated spherical $t$-designs with optimal order is further shown.

{Another closely related notion is the finite sequences in $\mathbb{C}^d$ named \emph{projective $t$-designs} \cite{iverson2021note} (also called \emph{spherical $(t,t)$-designs} \cite{waldron2017sharpening}). For finite sequences in which vectors are all unit norms, such designs have several equivalent characterizations stating that these designs are finite sequences that give equality in the Welch bounds, induce unit norm tight frames in the space of symmetric $t$-tensors on $\mathbb{C}^d$, and integrate all homogeneous $t$-degree polynomials in $z$ and $\bar{z}$ on the complex unit sphere by taking discrete sums as in spherical $t$-designs \cite{waldron2017sharpening}. Consequently, results from tight frame completion \cite[Proposition 21]{casazza2013introduction} (see also \cite{feng2006generation,massey2008tight}) can be used to discuss the number of extra vectors for extending a given finite sequence to a projective $t$-design, which is similar to the problem of extending a given point set to a spherical $t$-design.  For unit norm finite sequences, the vectors integrate all homogeneous $2t$-degree (and all $2s$ degrees, $2\leq 2s<2t$) polynomials on the real unit sphere by taking discrete sums. However, for $t>1$, it is impossible for these sequences to be equivalently characterized as tight frames in certain spaces of symmetric tensors \cite[Example 4.1]{waldron2017sharpening}. Thus, results from frame completion can not be applied to the real case unless $t=1$.}

{For a $t_1$-design $\mathcal{Y}$ with $t_1 < t$,  if it can be extended into a $t$-design by adding another point set $\mathcal{X}$, then we say that these $t_1$-design $\mathcal{Y}$ and $t$-design $\mathcal{X} \cup \mathcal{Y}$ are nested.} Given the above discussion, it is natural to ask the following questions.
\begin{itemize}
\item[{\rm Q1)}] Do nested spherical designs exist?

\item[{\rm Q2)}] If nested spherical designs exist, then how many points are needed to extend a $t_1$-design $(t_1< t)$ to a $t$-design? 

\item[{\rm Q3)}] What is the optimal order for the number of points needed to extend a $t_1$-design $(t_1 < t)$ to a $t$-design? 
\end{itemize}

To the best of our knowledge, such a concept of nested spherical designs and the above theoretical questions on nested spherical designs have not been considered in the literature. In this paper, we answer the first two questions completely and briefly discuss Q3.

For Q1, given a spherical $t_1$-design $\ptset{Y}_M:=\{y_1,\dots,y_M\}\subset \mathbb{S}^d$ with $t_1 <  t$, we prove that for sufficiently large $N$, there exist a point set $\ptset{X}_N= \{x_1,\dots,x_N\}\subset \mathbb{S}^d$ such that $\ptset{X}_N\cup\ptset{Y}_M$ is a spherical $t$-design, which leads to the nested spherical designs $\ptset{Y}_M \subset (\ptset{X}_N\cup\ptset{Y}_M)$. In fact, we prove a stronger statement with the given point set being an arbitrary point set $\ptset{Y}_M\subset\mathbb{S}^d$. The proof of this statement is a simple extension of the result in \cite{ostd}.

For Q2, we provide more precise estimates of $N$. The estimates of $N$ are based on the estimates of the linear functionals on a compact subset of polynomial spaces. Eventually, we conclude that for a spherical $t_1$-design $\ptset{Y}_M$ with the optimal order $M\sim t_1^d$, an upper bound of the minimal $N+M$ is of order $t^{2d+1}$ for the nested extension $\ptset{X}_N\cup \ptset{Y}_M$ to be a spherical $t$-design with $t_1 < t$.

In contrast to the optimal order $t^d$ for a spherical $t$-design, our result for Q2 shows that an upper bound of $N+M$ is of order $t^{2d+1}$.  It is interesting to know whether there exists any (non-trivial) nested spherical $t$-design of optimal order $t^d$. For Q3, we conjecture that the optimal order of $N+M$ is $t^d$ given some $t_1$-designs of order $t_1^d$ and give a brief discussion on the case where the constants in the orders are required to be identical.

This paper is organized as follows. Preliminaries are provided in Section \ref{sec:pre}. Existence, estimates, and a discussion of optimal order are given in Section \ref{sec:existence}, \ref{sec:estimates}, and \ref{sec:discuss}, respectively. Conclusion and final remarks are given in Section \ref{sec:conclude}. Some proofs are provided in the Appendix.

\section{Preliminaries}\label{sec:pre}
Let $\mathcal{P}_t:=\mathcal{P}_t(\mathbb{S}^d)$ be the Hilbert space of polynomials $P$ on $\mathbb{S}^d$ of degree at most $t$ such that
\[
	\int_{\mathbb{S}^d}P(x)d{\mu_d}(x)=0,
\]
and equipped with the usual inner product
\[
	\langle P,Q\rangle :=  \int_{\mathbb{S}^d}P(x)Q(x)d{\mu_d}(x).
\]
$\mathcal{P}_t$ is a finite dimension space spanned by the (real-valued) spherical harmonics (\cite{muller2006spherical,bondarenko2010spherical,dai2013approximation}) 
\[
\{Y_{\ell,k} \setsep \ell = 1,\dots,t,k\allowbreak = 1,\dots,D(\ell,d)\},
\]  
where $D(\ell,d)$ is defined as 
\[
	D(\ell,d) := \left\{\begin{aligned}
	\frac{2\ell+d-1}{\ell+d-1}\binom{\ell+d-1}{\ell} &\text{ if } \ell\geq 1,\\
	1 \ \ \ \ \ \ \ \ \ \ \ \ \ \ &\text{  if } \ell =0. \\
	\end{aligned}\right.
\]
We use $D_t$ to denote the dimension of $\mathcal{P}_t$. Note that $\sum_{i=0}^\ell D(i,d) = D(\ell,d+1)$ \cite{sloan2009variational} and thus $D_t = \sum_{i=1}^t D(i,d) = D(t,d+1)-1$. The $0$-degree spherical harmonic $Y_{0,1}\equiv \omega_d^{-1/2}$ is not included in $\mathcal{P}_t$ with
\[
	\omega_d := \int_{\mathbb{S}^d} d{\lambda_d(x)} = \frac{2\pi^{(d+1)/2}}{\Gamma((d+1)/2)},
\]
where  $\Gamma$ is the usual Gamma function and $\lambda_d$ is the (unnormalized) Lebesgue measure of $\mathbb{S}^d$.   
Note that $\omega_d$ is the surface area of $\mathbb{S}^d$ and $d\mu_d =\omega_d^{-1}d\lambda_d$.  Moreover, the $Y_{\ell,k}$'s are orthonormal with respect to the inner product
\[
	(P,Q):= \int_{\mathbb{S}^d}P(x)Q(x)d{\lambda_d}(x)=\omega_d\cdot \langle P,Q \rangle.
\]
In this paper,  the $L^1$-, $L^2$- and $L^\infty$-norms with respect to $\lambda_d$ are denoted as $\Vert \cdot \Vert_1$, $\Vert \cdot \Vert_2$ and  $\Vert \cdot \Vert_\infty$, respectively. 

The addition formula of spherical harmonics \cite{muller2006spherical} states that
\begin{equation}\label{at}
	\sum_{k=1}^{D(\ell,d)}Y_{\ell,k}(x)Y_{\ell,k}(y) = \frac{D(\ell,d)}{\omega_d}P_{\ell,d}(\langle x, y\rangle), \quad x,y \in \mathbb{S}^d,
\end{equation}
where $P_{\ell,d}$ is the Legendre polynomial of $\ell$-degree and dimension $d$, and $\langle x,y\rangle$ is the usual inner product for $x,y\in\mathbb{R}^{d+1}$. It is straightforward to verify that
\[
	k_t(x,y) := \omega_d\sum_{\ell=1}^t \sum_{k=1}^{D(\ell,d)} Y_{\ell,k}(x)Y_{\ell,k}(y), \quad x,y \in \mathbb{S}^d, 
\]
is a reproducing kernel of $\mathcal{P}_t$ under $\langle \cdot,\cdot \rangle$.
We obtain from (\ref{at}) that 
\[
	\sum_{k=1}^{D(\ell,d)}\left(Y_{\ell,k}(y)\right)^2 = \frac{D(\ell,d)}{\omega_d},\quad y\in \mathbb{S}^d,
\]
by setting $x=y$ and using $P_{\ell,d}(1)\equiv1$. 

For any $\ell$-degree spherical harmonics $Q$, we have the orthonormal expansion
\[
	Q = \sum_{k=1}^{D(\ell,d)}( Q,Y_{\ell,k} ) Y_{\ell,k}.
\]
Therefore,
\[
	|Q(x)|^2 \leq \sum_{k=1}^{D(\ell,d)} ( Q,Y_{\ell,k} )^2 \sum_{k=1}^{D(\ell,d)} \left(Y_{\ell,k}(y)\right)^2 = \Vert Q \Vert_2^2\frac{D(\ell,d)}{\omega_d}, \quad x\in\mathbb{S}^d. 
\]
and thus
\begin{equation}\label{nieq1}
\Vert Q\Vert_\infty \leq \Vert Q \Vert_2 \left(\frac{D(\ell,d)}{\omega_d}\right)^{1/2}.
\end{equation}

By the Riesz representation theorem, for each point $x\in \mathbb{S}^d$, there exists a unique polynomial $G_x\in\mathcal{P}_t$ such that
\begin{equation}\label{Eq1}
	\langle G_x,Q\rangle= Q(x)\text{ for all }Q\in\mathcal{P}_t. 
\end{equation}
By the reproducing kernel,  such a $G_x$ is explicitly given by
\begin{equation}\label{rpk}
	G_x(\cdot) = \omega_d\sum_{\ell=1}^t \sum_{k=1}^{D(\ell,d)} Y_{\ell,k}(x)Y_{\ell,k}(\cdot) = k_t(x,\cdot).
\end{equation}
Then it is obvious that  $\ptset{X}_N=\{x_1,\dots,x_N\}\subset \mathbb{S}^d$ is  a spherical $t$-design if and only if
\[
	G_{x_1} + \dots + G_{x_N} =0.
\]

The gradient of a differentiable function $f:\mathbb{R}^{d+1}\to \mathbb{R}$  is denoted by
\[
	\frac{\partial f}{\partial x}:=\left( \frac{\partial f}{\partial \xi_1},\dots,\frac{\partial f}{\partial \xi_{d+1}}\right),\quad x=(\xi_1,\dots,\xi_{d+1}).
\]
For a polynomial $Q\in\mathcal{P}_t$, we define the {\it spherical gradient} by
\[
	\nabla Q(x):= \frac{\partial}{\partial x}\left(Q\left(\frac{x}{|x|}\right)\right).
\] 
The $i$th component of $\nabla Q$ is denoted as $\nabla_i Q:=\frac{\partial Q}{\partial \xi_i}$, $1\leq i \leq d+1$. Note that $\int_{\mathbb{S}^d} |\nabla Q(x)|d\mu_d(x)$ is a norm on $\mathcal{P}_t$. 

For a partition $\mathcal{R}:=\{R_1,\dots,R_N\}$, where $R_i \subset \mathbb{S}^d$ are closed sets such that $\cup_{i=1}^N R_i= \mathbb{S}^d$  and $\mu_d(R_i\cap R_j)=0$ for all $1\leq i < j\leq N$, it is called \emph{area-regular} if $\mu_d(R_i) = 1/N$ for all $1\le i\le N$ and the \emph{partition norm} of $\mathcal{R}$ is denoted as
\[
	\Vert \mathcal{R} \Vert:= \underset{R\in \mathcal{R}}{\max}\text{ diam}(R),
\]
where $\text{diam}(R)$ is the maximum geodesic distance between two points in ${R}$. Each {member} $R_i$ of $\mathcal{R}$ is called a cell of $\mathcal{R}$.

We say that $N$ is of order $h(t_1,t,d)$ for some function $h$ depending on  $t_1,t,d$ if  $N = C_d\cdot h(t_1,t,d)$ for some constant $C_d$ depending only on $d$ but not the others.  The notion $N\sim t^d$ indicates equivalence relation $C_1 t^d\le N\le C_2 t^d$ with constants $C_1, C_2$ possibly depending on $d$ but not $t$.

\section{Existence of nested spherical designs}\label{sec:existence}

In this section, we prove the existence of a spherical $t$-design containing an arbitrarily prescribed point set $\ptset{Y}_M=\{y_1,\dots, y_M\}\subset \mathbb{S}^d$, which implies the existence of nested spherical designs.

We first state a fundamental theorem from the Brouwer degree theory \cite[Theorem 1.2.9]{cho2006topological} and some key results from \cite{ostd}. Our proof of the existence of nested spherical designs relies on such results.

\renewcommand*{\thetheorem}{\Alph{theorem}}
\begin{theorem}[\cite{cho2006topological}]\label{basic}
Let $f: \mathbb{R}^n \to \mathbb{R}^n$ be a continuous mapping and $\Omega\subset \mathbb{R}^n$ be an open bounded subset with boundary $\partial\Omega$ such that $0\in \Omega$. If $\langle x,f(x)\rangle>0$ for all $x\in\partial\Omega$, then there exists $x\in\Omega$ satisfying $f(x)=0$.
\end{theorem}

Since $\mathcal{P}_t$ is a finite-dimensional space with a norm: $\int_{\mathbb{S}^d} |\nabla P(x)|d\mu_d(x)$, $P\in\mathcal{P}_t$, we can consider 
a subset $\Omega$ of $\mathcal{P}_t$ defined by
\begin{equation}\label{Eq2}
	\Omega:= \left \{ P\in \mathcal{P}^t {\Biggr{|}} \int_{\mathbb{S}^d} |\nabla P(x)|d\mu_d(x)<1\right\}.  
\end{equation}
Then $\Omega$ is open and bounded in $\mathcal{P}_t$ containing $0$ and
\begin{equation}\label{def:pOmega}
\partial\Omega=\left \{P\in\mathcal{P}^t {\Biggr{|}} \int_{\mathbb{S}^d} |\nabla P(x)|d\mu_d(x)=1\right\}.
\end{equation}
One can apply Theorem \ref{basic} to $\Omega$ and deduce the following result.

\renewcommand*{\thecorollary}{\Alph{corollary}}
\begin{corollary}[\cite{ostd}]\label{cor:A}
Suppose  $F:\mathcal{P}_t \to (\mathbb{S}^d)^N$ with $F(P) = (x_1(P),\ldots, x_N(P))$ is a continuous mapping such that for all $P\in\partial\Omega$,
\[
	\sum^N_{i=1}P(x_i(P)) = \left \langle \sum^N_{i=1}G_{x_i(P)},P\right \rangle >0.
\]
Then there exists a spherical $t$-design in $\mathbb{S}^d$ consisting of $N$ points. 
\end{corollary}
\renewcommand*{\thecorollary}{\arabic{corollary}}
\setcounter{corollary}{0}

Indeed, if such an $F$ exists, then one can define a composition mapping 
\[
f=L\circ F:\mathcal{P}_t \to \mathcal{P}_t
\] 
with  $L:(\mathbb{S}^d)^N\to \mathcal{P}_t$ being given by $L(x_1,\dots,x_N)= G_{x_1}+\dots+G_{x_N}$. Then,
\[
	\langle P,f(P)\rangle = \sum^N_{i=1}P(x_i(P))
\]
for each $P\in \mathcal{P}_t$. Thus, applying Theorem~\ref{basic} to the mapping $f$, the vector space $\mathcal{P}_t$, and the subset $\Omega$ defined by (\ref{Eq2}), we obtain that $f(Q)=0$ for some $Q\in \mathcal{P}_t$. Hence, by (\ref{Eq1}), the components of $F(Q) = (x_1(Q),\dots,x_N(Q))$ form a spherical $t$-design in $\mathbb{S}^d$ consisting of $N$ points.  

The proofs of \cite[Lemma1, Theorem 1]{ostd}  give the following key result on the existence of $f=L\circ F$ and the estimate of $\langle P,f(P)\rangle= \sum^N_{i=1}P(x_i(P))$, which concludes that the optimal order of a spherical $t$-design is $t^d$. 

\begin{theorem}[\cite{ostd}]\label{thm2}
There exist two constants $C_d$ and $r_d$ depending only on $d$ for each $d\in \mathbb{N}$ and a continuous mapping $F: \mathcal{P}_t \to (\mathbb{S}^d)^N$ with $F(P)=(x_1(P),\dots,x_N(P)) $ such that when $N\geq C_dt^d$, we have
\begin{equation}
	\sum^N_{i=1}P(x_i(P)) > N\left(\frac{1}{6\sqrt{d}}\frac{r_d}{3t}-\frac{r_d}{18\sqrt{d}t}\right)=0,\quad \forall P\in \partial \Omega.
\end{equation}
\end{theorem}

However, to prove the existence of nested spherical designs,  we need a theorem that is slightly different from Theorem \ref{thm2}. In Theorem \ref{thm2}, the constant $18$ of $\frac{r_d}{18\sqrt{d}t}$ is directly related to the constant $C_d$ and one can see from the proof of in \cite{ostd} that changing $C_d$ would not affect the first term $\frac{1}{6\sqrt{d}}\frac{r_d}{3t}$. Thus, we can assume a larger constant $C_{1,d}:=2C_d$, which is still a constant that is only related to $d$, and as a result, the constant $18$ becomes $36$. We have the following modified theorem.

\renewcommand*{\thetheorem}{\arabic{theorem}}
\setcounter{theorem}{0}
\begin{theorem}\label{r1}
There exist two constants $C_{1,d}$ and $C_{2,d}$ depending only on $d$ for each $d\in \mathbb{N}$ and a continuous mapping $F: \mathcal{P}_t \to (\mathbb{S}^d)^N$ with  $F(P)=(x_1(P),\dots,x_N(P)) $ such that when $N\geq C_{1,d}t^d$, we have
\begin{equation}
	\sum^N_{i=1}P(x_i(P)) > NC_{2,d}t^{-1}> 0,\quad \forall P\in \partial \Omega.
\end{equation}
\end{theorem}

The proof of Theorem \ref{r1} is almost a verbatim proof of \cite[Theorem 1]{ostd} with only some modifications of constants. For completeness, we include it in the Appendix. Using Theorem~\ref{r1}, we can prove the following result for a spherical $t$-design containing an arbitrarily prescribed point set.

\begin{theorem}\label{thm3}
Let $t\in\mathbb{N}$ and $\ptset{Y}_M=\{y_1,\dots,y_M\}\subset\mathbb{S}^d$ be a set of $M$ points on the sphere. Then, for sufficiently large $N$, there exists a point set $\ptset{X}_N=\{x_1,\dots,x_N\}\subset \mathbb{S}^d$ such that the union $\ptset{X}_N\cup \ptset{Y}_M$ is a spherical $t$-design.
\end{theorem}

\begin{proof}
By Corollary~\ref{cor:A}, it is sufficient to consider the existence of a continuous map $F: \mathcal{P}_t \to (\mathbb{S}^d)^N$ with $F(P)=(x_1(P),\dots,x_N(P))$ such that its composition with a continuous mapping $L':(\mathbb{S}^d)^N \to \mathcal{P}_t$
\[
	(x_1,\dots,x_N) \overset{L'}{\to} G_{x_1}+\dots +G_{x_N} + G_{y_1}+\dots+G_{y_M},
\]
i.e. $g:=L'\circ F:\mathcal{P}_t \to \mathcal{P}_t$ satisfies
\begin{gather*}
	\sum^N_{i=1}P(x_i(P)) + \sum^M_{j=1}P(y_j) = \langle P,g(P)\rangle = \left\langle \sum^N_{i=1}G_{x_i(P)},P\right \rangle +\left \langle \sum^{M}_{j=1}G_{y_i},P\right \rangle >0, 
\end{gather*}
for all $P\in\partial \Omega$ with $\Omega$ being defined as in \eqref{Eq2}.

From Theorem~\ref{r1}, for $N\geq C_{1,d}t^d$, there exists a continuous mapping $F: \mathcal{P}_t \to (\mathbb{S}^d)^N$ with $F(P)=(x_1(P),\ldots,x_N(P))$ such that
\[
	\sum^N_{i=1}P(x_i(P)) = \left\langle \sum^N_{i=1}G_{x_i(P)},P\right \rangle > NC_{2,d}t^{-1}>0,\quad \forall P\in \partial \Omega.
\]
Since $\partial\Omega$ is a compact subset of $\mathcal{P}_t$, the quantity
\[
	\underset{P\in\partial\Omega}{\inf} \left \langle \sum^{M}_{j=1}G_{y_i},P\right\rangle
\]
is finite and we denote it as $m$. Thus, when $N$ is sufficiently large, we have 
\[
\left\langle \sum^N_{i=1}G_{x_i(P)},P\right \rangle +\left \langle \sum^{M}_{j=1}G_{y_i},P\right \rangle > NC_{2,d}t^{-1} + m >0,
\]
for all $P\in\partial \Omega$, which completes the proof. 
\end{proof}

The existence of nested spherical designs, as shown by the following corollary, follows directly from Theorem~\ref{thm3}.

\begin{corollary}\label{cor:nested}
Let $t_1, t\in\mathbb{N}$  be two integers such that $t_1 <  t$ and $\ptset{Y}_M=\{y_1,\allowbreak \dots,y_M\}\subset \mathbb{S}^d$ be a spherical $t_1$-design. Then, for sufficiently large $N$, there exists a point set $\ptset{X}_N=\{x_1,\dots,x_N\}\subset \mathbb{S}^d$ such that $\ptset{X}_N\cup\ptset{Y}_M$ is a spherical $t$-design.
\end{corollary}

\section{Estimates of nested spherical designs}\label{sec:estimates}

As indicated in the previous section, to estimate the number of required points for extending a given point set $\mathcal{Y}_M:= \{y_1,\dots,y_M\}$ to a spherical $t$-design, we need to estimate the quantity 
$
	\inf_{P\in\partial \Omega}\left \langle \sum^{M}_{j=1}G_{y_i},P\right\rangle
$
from below. 
Since 
\[
	\left |\left \langle \sum^{M}_{j=1}G_{y_i},P\right\rangle \right| \leq \omega_d^{-1}\left\Vert \sum^{M}_{j=1}G_{y_i}\right\Vert_2 \Vert P\Vert_2, 
\]
it can be reduced to the estimates of the $L^2$-norms of $P \in \partial \Omega$ and $\sum^{M}_{j=1}G_{y_i}$ with arbitrary $\ptset{Y}_M$. 

First, for $\|P\|_2$, we have the following estimate.

\begin{lemma}\label{p_norm}
For any $P\in\partial \Omega$, it holds that
\[
	\Vert P \Vert_2 \leq \sqrt{\frac{d+1}{d}\omega_dD(t+1,d+1)}.
\]
\end{lemma}

\begin{proof}
For an $\ell$-degree spherical harmonic $Y_{\ell,k}$, the components of $\nabla Y_{\ell,k}$ are linear combinations of spherical harmonics of degree $\ell+1$ and $\ell-1$ (\cite[Chapter 12, (12.3.2)]{freeden1998constructive}). Hence, for a $t'$-degree polynomial $P\in\mathcal{P}_t$, we have $\nabla_i P = \sum_{j=0}^{t'+1}Q_{i,j}$, where $Q_{i,j}$ is either zero or a spherical harmonic of {degree $j$}. Note that by the orthonormality of the spherical harmonics, we also have $\Vert \nabla_i P\Vert_2^2= \sum_{j=0}^{t'+1}\Vert Q_{i,j}\Vert_2^2$.

Applying the interpolation inequality and (\ref{nieq1}), we obtain
\begin{align*}
	\Vert \nabla_i P\Vert_2 &\leq  \Vert \nabla_i P\Vert_1^{1/2}\Vert \nabla_i P \Vert_\infty^{1/2}
\leq  \Vert \nabla_i P\Vert_1^{1/2}\left(\sum_{j=0}^{t'+1}\Vert Q_{i,j} \Vert_\infty\right)^{1/2}\\
&\leq  \Vert \nabla_i P\Vert_1^{1/2}\left(\sum_{j=0}^{t'+1}\Vert Q_{i,j} \Vert_2\left(\frac{D(j,d)}{\omega_d}\right)^{1/2} \right)^{1/2}\\
&\leq  \Vert \nabla_i P\Vert_1^{1/2}\left(\sum_{j=0}^{t'+1}\Vert Q_{i,j}\Vert^2_2\right)^{1/4}\left(\sum_{j=0}^{t'+1}\frac{D(j,d)}{\omega_d}\right)^{1/4}\\
&=\Vert \nabla_i P\Vert_1^{1/2}\Vert \nabla_i P\Vert_2^{1/2}\left(\frac{D(t'+1,d+1)}{\omega_d}\right)^{1/4}\\
&\leq\Vert \nabla_i P\Vert_1^{1/2} \Vert \nabla_i P\Vert_2^{1/2}\left(\frac{D(t+1,d+1)}{\omega_d}\right)^{1/4}.
\end{align*}
Thus, we deduce that
\[
	\sum_{i=1}^{d+1}\Vert \nabla_i P \Vert_2 \leq \left(\frac{D(t+1,d+1)}{\omega_d}\right)^{\frac{1}{2}}\sum_{i=1}^{d+1}\Vert \nabla_i P\Vert_1.
\]
Note that 
\begin{align}\label{nieq2}
	\omega_d^{-1}\sum_{i=1}^{d+1}\Vert \nabla_i P\Vert_1
	 = \int_{\mathbb{S}^d}\sum_{i=1}^{d+1} |\nabla_i P(x)| d{\mu_d}(x)
	\leq (d+1)^{1/2} \int_{\mathbb{S}^d} |\nabla P(x)|d{\mu_d}(x).
\end{align}
Combined with (\ref{nieq2}) and the Poincar\'e inequality on $\mathbb{S}^d$ \cite{li2016poincar}, that is, $\Vert P \Vert_2^2 \leq d^{-1} \sum_{i=1}^{d+1}\Vert \nabla_i P \Vert_2^2$, we obtain for all $P\in \partial\Omega$, 
\begin{align*}
	\Vert P \Vert_2 &\leq d^{-1/2} \left(\sum_{i=1}^{d+1}\Vert \nabla_i P \Vert_2^2\right)^{1/2} \leq d^{-1/2} \sum_{i=1}^{d+1}\Vert \nabla_i P \Vert_2\\
&\leq d^{-1/2}\left(\frac{D(t+1,d+1)}{\omega_d}\right)^{\frac{1}{2}}\sum_{i=1}^{d+1}\Vert \nabla_i P\Vert_1\\
&\leq d^{-1/2}\left(\frac{D(t+1,d+1)}{\omega_d}\right)^{\frac{1}{2}}\omega_d(d+1)^{1/2} \int_{\mathbb{S}^d} |\nabla P(x)|d{\mu_d}(x) \\
&= \sqrt{\frac{d+1}{d}\omega_dD(t+1,d+1)}, 
\end{align*}
which completes the proof.
\end{proof}

\begin{remark}
{\rm We can regard $\sqrt{\frac{d+1}{d}\omega_dD(t+1,d+1)}$ as an estimate of one of the equivalence constants between the norm $\Vert P\Vert_2$ and the norm $\int_{\mathbb{S}^d} |\nabla P(x)|d\mu_d(x)$ on $\mathcal{P}_t$.}
\end{remark}

Next, we estimate $\left\Vert \sum^{M}_{j=1}G_{y_i}\right\Vert_2$. Recall that $D_t = D(t,d+1)-1$.
\begin{lemma}\label{g_norm} 
For $\ptset{Y}_M=\{y_1,\ldots, y_M\}\subset \mathbb{S}^d$ and $\{G_{y_i}\setsep y_i\in\ptset{Y}_M\} \subset\mathcal{P}_t$, it holds that
\begin{equation}\label{g1}
\left\Vert \sum^{M}_{j=1}G_{y_i}\right\Vert_2 \leq M\sqrt{\omega_d D_t }.
\end{equation}
Moreover, if $\ptset{Y}_M$ is a spherical $t_1$-design with $t_1 < t$, then
\begin{equation}\label{g2}
	\left\Vert \sum^{M}_{j=1}G_{y_i}\right\Vert_2 \leq M\sqrt{\omega_d(D_t-D_{t_1})}.
\end{equation}
\end{lemma}

\begin{proof}
By (\ref{rpk}), the square of $L^2$-norm of $\sum^{M}_{j=1}G_{y_i}$ is 
\begin{equation}\label{norm1}
	\left\|\sum^{M}_{j=1}G_{y_i}\right \|_2^2 =\omega_d^2\sum_{\ell=1}^t \sum_{k=1}^{D(\ell,d)}\left(\sum_{j=1}^MY_{\ell,k}(y_j)\right)^2.
\end{equation}
When $\mathcal{Y}_M$ happens to be a spherical $t_1$-design with $t \ge t_1$, it  reduces to
\begin{equation}\label{norm2}
	 \left\|\sum^{M}_{j=1}G_{y_i}\right\|_2^2 = \omega_d^2\sum_{\ell=t_1+1}^t \sum_{k=1}^{D(\ell,d)}\left(\sum_{j=1}^MY_{\ell,k}(y_j)\right)^2.
\end{equation}
Now, by the Cauchy–Schwartz inequality, we have
\begin{equation}\label{ieq1}
\begin{aligned}
	\sum_{\ell=1}^t \sum_{k=1}^{D(\ell,d)}\left(\sum_{j=1}^MY_{\ell,k}(y_j)\right)^2 
	& \leq M \sum_{j=1}^M\sum_{l=1}^t\sum_{k=1}^{D(\ell,d)} Y^2_{\ell,k}(y_j)
	\\&= M^2 \frac{\sum_{\ell=1}^t D(\ell,d)}{\omega_d} 
	\\&= M^2 \frac{D_t}{\omega_d}.
	\end{aligned}
\end{equation}
Similarly, we have
\begin{gather}\label{ieq2}
	\sum_{\ell=t_1+1}^t \sum_{k=1}^{D(\ell,d)}\left(\sum_{j=1}^MY_{\ell,k}(y_j)\right)^2 \leq M^2\frac{\sum_{\ell=t_1+1}^t D(\ell,d)}{\omega_d} = M^2 \frac{D_t-D_{t_1}}{\omega_d}.
\end{gather}
Therefore, the inequalities \eqref{g1} and \eqref{g2} hold. 
\end{proof}

\begin{remark}
{\rm
The estimate (\ref{ieq1}) is essentially the upper bound of the quantity $A_{L, N}$ in \cite[Theorem 3]{sloan2009variational}, where $L, N$ correspond to our $t, M$, respectively, and the upper bound differs by a constant $M^2$. In fact $A_{t,M}(\ptset{Y}_M) =\left\|\frac{1}{M}\sum^{M}_{j=1}G_{y_i}\right \|_2^2$.
}
\end{remark}

By Lemmas \ref{p_norm} and \ref{g_norm}, we have the following estimate for $\inf_{P\in\partial \Omega}\left \langle \sum^{M}_{j=1}G_{y_i},P\right\rangle$.

\begin{lemma}\label{Gp_norm}
For $\ptset{Y}_M=\{y_1,\ldots, y_M\}\subset \mathbb{S}^d$ and $\{G_{y_i}\setsep y_i\in\ptset{Y}_M\} \subset\mathcal{P}_t$, it holds that
\[
\inf_{P\in\partial \Omega}\left \langle \sum^{M}_{j=1}G_{y_i},P\right\rangle
\ge -M\sqrt{\frac{d+1}{d}D(t+1,d+1)D_t}.
\]
Moreover, if $\ptset{Y}_M$ is a spherical $t_1$-design with $t_1 <  t$, then,
\[
\inf_{P\in\partial \Omega}\left \langle \sum^{M}_{j=1}G_{y_i},P\right\rangle
\ge -M\sqrt{\frac{d+1}{d}D(t+1,d+1)(D_t-D_{t_1})}.
\]
\end{lemma}

We are now able to state the existence of extensions to spherical $t$-designs with estimates of the required numbers of points.
\begin{theorem}\label{thm4}
Let $t\in\mathbb{N}$, $C_{1,d}, C_{2,d}$ be constants as in Theorem~\ref{r1}, and  $C_{3,d}:=\sqrt{\frac{d+1}{d}}$. For any given  $\ptset{Y}_M=\{y_1,\ldots,y_M\}\subset \mathbb{S}^d$ and 
\begin{equation}\label{bd:N1}
N \geq \max\left(C_{1,d}t^d,C^{-1}_{2,d}C_{3,d}Mt\sqrt{D(t+1,d+1)D_t}\right),
\end{equation}  
there exists $\ptset{X}_N=\{x_1,\ldots,x_N\}\subset \mathbb{S}^d$  such that $\ptset{X}_N\cup\ptset{Y}_M$ is a spherical $t$-design. Moreover, if  $\ptset{Y}_M$ is a spherical $t_1$-design with $t_1  <  t$, then the lower bound of $N$ becomes 
\begin{equation}\label{bd:N2}
N\ge \max\left(C_{1,d}t^d,C^{-1}_{2,d}C_{3,d}Mt\sqrt{D_t-D_{t_1}}\allowbreak\sqrt{D(t+1,d+1)}\right).
\end{equation}
\end{theorem}

\begin{proof}
By Theorem \ref{thm3} and Lemma \ref{Gp_norm}, we have that
\[
\left\langle \sum^N_{i=1}G_{x_i(P)},P\right \rangle +\left \langle \sum^{M}_{j=1}G_{y_i},P\right \rangle > NC_{2,d}t^{-1} - C_{3,d}M\sqrt{D(t+1,d+1)D_t} \geq 0
\]
{implying} the existence of spherical $t$-design formed by $\ptset{X}_N\cup\ptset{Y}_M$. This gives the bound of $N$ in \eqref{bd:N1}. The bound $N$ in \eqref{bd:N2} for the case of $\ptset{Y}_M$ being a spherical $t_1$-design is similar.
\end{proof}

By the result of Theorem~\ref{thm4}, we have the following corollary giving the order of $N$ for nested spherical designs.
\begin{corollary}\label{cor_s4}
 Let  $\ptset{Y}_M:=\{y_1,\dots,y_M\}$ be a spherical $t_1$-design of optimal order $t_1^d$ and $t_1 <  t$. Then  $N+M$ is of order at most $t^{2d+1}$ for the existence of $\ptset{X}_N:=\{x_1,\dots,x_N\}$ such that $\ptset{X}_N\cup \ptset{Y}_M$ is a spherical $t$-design.
\end{corollary}

\begin{proof}
Since $D(t,d) \sim (t+1)^{d-1}$ \cite{sloan2009variational}, $(t+1)^{d-1} \sim (t-1)^{d-1} \sim t^{d-1}$, and 
$
	D_t = D(t,d+1)-1 < D(t,d+1) \sim t^d,
$
we have $\sqrt{D(t+1,d+1)D_t}\le Ct^d$ for some constant $C$ and by Theorem~\ref{thm4}, we deduce that
\[
	(N+M)\sim Mt\sqrt{D(t+1,d+1)D_t} + M \le C_1 (t^d\cdot t\cdot t^{d} + t^d) \le C_2 t^{2d+1}, 
\]
where $C_1,C_2$ are constants depending only on $d$.
\end{proof}


\section{Discussion on the optimal order}\label{sec:discuss}
In contrast to the optimal order $t^d$ for a spherical $t$-design, our result in Corollary~\ref{cor_s4}  shows that an upper bound of the minimal $N+M$ 
in \emph{arbitrary} nested spherical $t$-designs with the $t_1$-design $\ptset{Y}_M$ of order $t_1^d$ is of order $t^{2d+1}$. This is, of course, a rough estimate. Moreover, the properties of $t_1$-designs are not involved. Following the setting in general spherical $t$-designs where the constant in the orders are identical for all $t$, it is thus natural to conjecture that the following statement is true.

\medskip

\noindent {\bf Conjecture.} \textit{For any $t_1,t\in \mathbb{N}$ such that $t_1 <   t$, there exist some spherical $t_1$-designs of order $t_1^d$ such that they can be extended to a spherical $t$-design of order $t^d$. Moreover, the constants (not depending on $t_1$ and $t$ but  $d$) in the orders are \textbf{identical} for any $t$ and $t_1$ satisfying $t_1 < t$.}

\medskip

If we specify the relation between $t_1$ and $t$, we can show that the problem actually becomes trivial. In fact, we have the following proposition.

\begin{proposition}\label{prop}
Given $m = p/q \in \mathbb{Q}$, $m > 1$, for any $t_1,t\in \mathbb{N}$ such that $mt_1=t$, there exist some spherical $t_1$-designs of order $t_1^d$ such that they can be extended to a spherical $t$-design of order $t^d$. Moreover, the constants (not depending on $t_1$ and $t$ but $d$ and $m$) in the orders are \textbf{identical} for any $t$ and $t_1$ satisfying $mt_1=t$.
\end{proposition}
\begin{proof}
Let $C := q^dC_d$ with $C_d$ as in Theorem \ref{thm2}. We start from a $t$-design with $N=Ct^d$ points (the existence is guaranteed by Theorem \ref{thm2}). Note that it is also a $t_1$-design with $N=m^dCt_1^d$ points. Thus, to obtain a $t$-design of which the order has the same constant, we need to extend this $t$-design into another $t$-design with $m^dCt^d$ points. Note that $N/q^d=C_dt^d$ is also large enough to guarantee a $t$-design. Since the union of $t$-designs is still a $t$-design, we add a $t$-design with $N/q^d$ points $(p^d-q^d)$ times. This gives a $t$-design with $m^dCt^d$ points. In this way, the order for both $t_1$ and $t$ are the same as  $t_1^d$ and $t^d$, respectively, while the constant is always $m^dC$. 
\end{proof}

Note that the assumption $m\in\mathbb{Q}$ is natural since both $t_1, t$ are integers. 
Compared with Proposition~\ref{prop}, the difficulty in proving the {conjecture} is that we need to find an identical constant for orders $t_1^d$ and $t^d$ without specifying the relation $mt_1=t,m\in\mathbb{Q}$, $m>1$ but only the general condition $t_1<t$. As a result, the dependency of $m$ in the constants should be removed. On the other hand, both Corollary \ref{cor_s4} and the conjecture are stated for any pair of $t_1,t$ such that $t_1<t$ without further specification. Compared with Corollary \ref{cor_s4} in which the constant in the order $t^{2d+1}$ is larger than the constant in $t^d$, the conjecture requires not only a smaller order ($t^d<t^{2d+1}$) but also zero difference between the constants. Hence, the conjecture is more restricted and thus could be more difficult to prove.

We remark that in practice, one would seek not only the optimal order but also the least possible constant in the order. Nonetheless,  the same constant in the optimal order and the stated existence result in Proposition~\ref{prop} gives some evidence and guidance for numerical experiments in practice: one may set an identical constant to find nested spherical $t$-designs with multi-nested structure, e.g. $t_1=2t_2=4t_3=\cdots$, and that each $t_i$-design is of order $t_i^d$ and is contained in the $t_{i-1}$-design.

\section{Conclusion and final remarks}\label{sec:conclude}

We have shown the existence and estimates of required points to form a spherical $t$-design given a fixed set of points. Consequently, in the case that the given point set is a spherical $t_1$-design with $t_1 <  t$ and the number of given points is of optimal order $t_1^d$, we provide an upper bound of the minimal total number of given points and required points, which is of order $t^{2d+1}$.

Nonetheless, these results do not conclude the optimal order of nested spherical designs. Corollary \ref{cor_s4} is stated for a given arbitrary spherical $t_1$-design (of optimal order $t_1^d$) that leads the estimate of order $t^{2d+1}$ of $N+M$. The estimate is pessimistic and can be regarded as a worst-case estimate. On the other hand, we have further discussed a conjecture concerning the constants in the optimal orders of $t_1$ and $t$-designs. We show that a proposition derived from the conjecture with an extra condition is trivial. However, it could be difficult to confirm or reject the original conjecture.

Finally, Theorem \ref{thm3} suggests an notion of optimality which requires an $t_1$-design to have the minimum value of $\langle \sum^{M}_{j=1}G_{y_i},P\rangle$ on $\partial \Omega$ among all $t_1$-designs.  It remains to investigate (at least numerically) whether properties of spherical designs such as being well-separated or well-conditioned \cite{an2010well} lead to better estimates or fewer points in numerical experiments.

\section*{Acknowledgements}
We would like to thank anonymous reviewers for their valuable comments and suggestions that greatly helped us to improve the presentation of this paper. This work was supported in part by the Research Grants Council of the Hong Kong Special Administrative Region, China, under Project CityU 11309122 and Project CityU 11302023.

\bibliographystyle{abbrv}
\bibliography{reference}

\appendix
\section{Theorems}

For the sake of being self-contained, as in \cite{ostd}, here we state some other theorems that will be used in the following proofs.

\renewcommand*{\thetheorem}{\Alph{theorem}}
\setcounter{theorem}{2}
\begin{theorem}[\cite{bourgain1988distribution,kuijlaars1998asymptotics}]\label{app_thm_1}
For each $N\in \mathbb{N}$, there exists an area-regular partition $\mathcal{R}=\{R_1,\dots,R_N\}$ with $\Vert \mathcal{R}\Vert \leq B_dN^{-1/d}$ for some constant $B_d$ large enough.
\end{theorem}

\begin{theorem}[\cite{mhaskar2001spherical}]\label{app_thm_2}
There exists a constant $r_d$ such that for each area-regular partition $\mathcal{R}=\{R_1,\dots,R_N\}$ with $\Vert \mathcal{R}\Vert <\frac{r_d}{m}$, each collection of points $x_i\in R_i(i=1,\dots,N)$, and each polynomial $P$ of total degree $m$, the inequality
\[
	\frac{1}{2}\int_{\mathbb{S}^d}|P(x)|d\mu_d(x)\leq \frac{1}{N}\sum_{i=1}^N|P(x_i)|\leq \frac{3}{2}\int_{\mathbb{S}^d}|P(x)|d\mu_d(x)
\]
holds.
\end{theorem}

\begin{theorem}[{\cite[Corollary 1]{ostd}}]\label{app_thm_3}
For each area-regular partition $\mathcal{R}=\{R_1,\dots,\allowbreak R_N\}$ with $\Vert \mathcal{R}\Vert <\frac{r_d}{m+1}$, each collection of points $x_i \in R_i(i=1,\dots,N)$, and each polynomial $P$ of total degree $m$,
\[
	\frac{1}{3\sqrt{d}}\int_{\mathbb{S}^d}|\nabla P(x)|d\mu_d(x) \leq \frac{1}{N}\sum_{i=1}^N|\nabla P(x_i)| \leq 3\sqrt{d}\int_{\mathbb{S}^d}|\nabla P(x)|d\mu_d(x).
\]
\end{theorem}

\section{Proof of Theorem \ref{r1}}\label{appendix:A}

\begin{proof}
For $d,t\in \mathbb{N}$, take $C_{1,d}>(108B_d/r_d)^d$, where $B_d$ is as in Theorem \ref{app_thm_1} and $r_d$ is as in Theorem \ref{app_thm_2} and $N\geq C_{1,d}t^d$. 
We take an area-regular partition $\mathcal{R}=\{R_1,\dots,R_N\}$ with
\begin{equation}\label{app1}
\Vert \mathcal{R}\Vert \leq B_dN^{-1/d}<\frac{r_d}{108dt}
\end{equation}
from Theorem \ref{app_thm_1}. For each $i=1,\dots,N$, we choose an arbitrary $\tilde{x}_i\in R_i$. 
Define $U:\mathcal{P}_t \times \mathbb{S}^d \to \mathbb{R}^{d+1}$ such that 
$
	U(P,w) = \frac{\nabla P(w)}{h_\varepsilon(|\nabla P(w)|)} 
$
where
$\varepsilon = \frac{1}{6\sqrt{d}}$ and 
\[
	h_{\varepsilon}(u):= \begin{cases}
	u\quad \text{if } u>\varepsilon,\\
	\varepsilon\quad \text{otherwise.}
\end{cases}
\]
For each $i = 1,\dots,N$, define $w_i:\mathcal{P}_t\times[0,\infty) \to \mathbb{S}^d$ be the map satisfying the differential equation
\begin{equation}\label{app2}
\frac{d}{ds}w_i(P,s)=U(P,w_i(P,s))
\end{equation}
with the initial condition
\[	
w_i(P,0) = \tilde{x}_i
\]
for each $P\in \mathcal{P}_t$. Each mapping $w_i$ has its values in $\mathbb{S}^d$ by {the} definition of spherical gradient. By the fact that the mapping $U(P,w)$ is Lipschitz continuous in both $P$ and $w$, each $w_i$ is well defined and continuous in both $P$ and $s$, where the metric on $\mathcal{P}_t$ is given by the inner product. 

The mapping $F:\mathcal{P}_t\to (\mathbb{S}^d)^N $ is defined as
\begin{equation}\label{app3}
F(P) = (x_1(P),\dots,x_N(P)):= \left(w_1(P,\frac{r_d}{3t}),\dots,w_N(P,\frac{r_d}{3t})\right).
\end{equation}
By definition, $F$ is continuous on $\mathcal{P}_t$. It remains to show that there exists a constant $C_{2,d}$ depending on $d$ such that 
\begin{equation}\label{app4}
	\sum^N_{i=1}P(x_i(P)) >  NC_{2,d}t^{-1}> 0,\quad \forall P\in \partial \Omega.
\end{equation}
Fix $P\in\partial\Omega$, i.e., 
\[
	\int_{\mathbb{S}^d}|\nabla P(x)|d\mu_d(x) = 1.
\]

For brevity we write $w_i(s)$ to replace $w_i(P,s)$. The Newton-Leibniz formula gives
\begin{equation}\label{app5}
\begin{aligned}
	\frac{1}{N}\sum_{i=1}^NP(x_i(P)) &= \frac{1}{N}\sum_{i=1}^NP(w_i(r_d/3t))\\
	&= \frac{1}{N}\sum_{i=1}^NP(\tilde{x}_i) + \int_0^{r_d/3t}\frac{d}{ds}\left[\frac{1}{N}\sum_{i=1}^NP(w_i(s))\right]ds.
\end{aligned}
\end{equation}
In the {remainder of this proof}, we will estimate the value
\[
 \left| \frac{1}{N}\sum_{i=1}^NP(\tilde{x}_i)\right|
\]
from above and the value
\[
	\frac{d}{ds}\left[ \frac{1}{N}\sum_{i=1}^NP(w_i(s))\right]
\]
from below for each $s\in [0,r_d/3t]$. 

We obtain
\begin{align*}
 \left| \frac{1}{N}\sum_{i=1}^NP(\tilde{x}_i)\right| &= \left| \sum_{i=1}^N \int_{R_i}P(\tilde{x}_i)-P(x)d\mu_d(x)\right| \leq \sum_{i=1}^N |P(\tilde{x}_i)-P(x)|d\mu_d(x)\\
&\leq \frac{\Vert \mathcal{R}\Vert}{N}\sum_{i=1}^N\underset{z\in \mathbb{S}^d:\text{dist}(z,\tilde{x}_i)\leq \Vert \mathcal{R}\Vert}{\max}|\nabla P(z)|,
\end{align*}
where $\text{dist}(z,\tilde{x}_i)$ denotes the geodesic distance between $z$ and $\tilde{x}_i$. Hence, for $z_i \in \mathbb{S}^d$ such that $\text{dist}(z_i,\tilde{x}_i)\leq \|\mathcal{R}\|$ and
\[
|\nabla P(z_i)| = \underset{z\in S^d:\text{dist}(z,\tilde{x}_i)\leq \Vert \mathcal{R}\Vert}{\max} |\nabla P(z)|,
\]
we have
\[
	\left| \frac{1}{N}\sum_{i=1}^NP(\tilde{x}_i)\right| \leq \frac{\Vert \mathcal{R}\Vert}{N}\sum_{i=1}^N|\nabla P(z_i)|.
\]

Define another area-regular partition $\mathcal{R}' = \{R'_1,\dots,R'_N\}$ by $R'_i=R_i\cup\{z_i\}$. It is obvious that $\Vert \mathcal{R}'\Vert \leq 2\Vert \mathcal{R}\Vert$. Combined with (\ref{app1}), we have $\Vert \mathcal{R}'\Vert \leq r_d/(54dt)$. Using {the} inequality in Theorem \ref{app_thm_3} to the partition $\mathcal{R}'$ and the collection of points $z_i$, we have
\begin{equation}\label{app6}
\left| \frac{1}{N}\sum_{i=1}^N P(\tilde{x}_i) \right| \leq 3\sqrt{d}\Vert \mathcal{R}\Vert \int_{\mathbb{S}^d}|\nabla P(x)|d\mu_d(x) < \frac{r_d}{36\sqrt{d}t}
\end{equation}
for any $P\in \partial \Omega$. On the other hand, the differential equation (\ref{app2}) implies
\begin{equation}\label{app7}
\begin{aligned}
\frac{d}{ds}\left[ \frac{1}{N}\sum_{i=1}^NP(w_i(s))\right] &= \frac{1}{N}\sum_{i=1}^{N}\frac{|\nabla P(w_i(s))|^2}{h_\varepsilon(|\nabla P(w_i(s))|)}\\
&\geq \frac{1}{N}\sum_{i:|\nabla P(w_i(s))|\geq \varepsilon} |\nabla P(w_i(s))|\\
&\geq \frac{1}{N}\sum_{i=1}^N|\nabla P(w_i(s))|-\varepsilon.
\end{aligned}
\end{equation}
Since 
\[
\left|\frac{\nabla P(w)}{h_\varepsilon(|\nabla P(w)|)}\right| \leq 1
\]
for each $w\in \mathbb{S}^d$, it follows again from (\ref{app2}) that $\left|\frac{dw_i(s)}{ds}\right|\leq 1$ and therefore
\[
	\text{dist}(\tilde{x}_i,w_i(s))\leq s.
\]

For each $s\in [0,r_d/3t]$, define another partition $\mathcal{R}''=\{R''_1,\dots,R''_N\}$ given by $R''_i = R_i\cup\{w_i(s)\}$. By (\ref{app1}), we have
\[
	\Vert \mathcal{R}''\Vert < \frac{r_d}{108dt}+\frac{r_d}{3t},
\]
and thus we can apply Theorem \ref{app_thm_3} to the partition $\mathcal{R}''$ and the collection of points $w_i(s)$. This and inequality (\ref{app7}) imply
\begin{equation}\label{app8}
\begin{aligned}
\frac{d}{ds}\left[ \frac{1}{N}P(w_i(s))\right] &\geq \frac{1}{N}\sum_{i=1}^N|\nabla P(w_i(s))| -\frac{1}{6\sqrt{d}}\\
&\geq \frac{1}{3\sqrt{d}}\int_{\mathbb{S}^d}|\nabla P(x)|d\mu_d(x)-\frac{1}{6\sqrt{d}}= \frac{1}{6\sqrt{d}}
\end{aligned}
\end{equation}
for each $P\in\partial\Omega$ and $s\in [0,r_d/3t]$.

Finally, equation (\ref{app5}) and inequalities (\ref{app6}) and (\ref{app8}) yield
\begin{equation}\label{app9}
\sum^N_{i=1}P(x_i(P)) > N\left(\frac{1}{6\sqrt{d}}\frac{r_d}{3t}-\frac{r_d}{36\sqrt{d}t}\right) > 0,\quad \forall P\in \partial \Omega.
\end{equation}
Let $C_{2,d}:=\frac{r_d}{36\sqrt{d}}$.
\end{proof}

%
%
%


\end{document}